\documentclass{siamonline1116}
\usepackage{amssymb}
\usepackage{amsmath}

\usepackage{subcaption}
\usepackage{wrapfig}

\usepackage{graphicx}
\usepackage{placeins}
\usepackage{fullpage}
\usepackage{xcolor}
\usepackage{tikz}

\usetikzlibrary{calc}
\usetikzlibrary{decorations.shapes}
\usetikzlibrary{decorations.pathreplacing}
\usetikzlibrary{decorations.pathmorphing}

\DeclareMathOperator{\aff}{aff}
\DeclareMathOperator{\conv}{conv}
\DeclareMathOperator{\init}{init}
\DeclareMathOperator{\NVol}{NVol}

\DeclareMathOperator{\Ht}{ht}

\newcommand{\C}{\mathbb{C}}
\newcommand{\R}{\mathbb{R}}
\newcommand{\Z}{\mathbb{Z}}

\newcommand{\imag}{\mathbf{i}}

\newtheorem{problem}{Problem Statement}
\newtheorem{remark}{Remark}

\begin{document}

\title{Counting equilibria of the Kuramoto model using birationally invariant intersection index}

\author{
    Tianran Chen%
    \thanks{Department of Mathematics and Computer Science,
        Auburn University at Montgomery,
        Montgomery AL USA}
    \and
    Robert Davis%
    \thanks{Department of Mathematics,
        Michigan State University,
        East Lansing, MI USA.}
    \and
    Dhagash Mehta%
    \thanks{Systems Department, United Technologies Research Center, East Hartford, CT, USA.}
}

\maketitle

\begin{abstract}
    Synchronization in networks of interconnected oscillators is a fascinating
    phenomenon that appear naturally in many independent fields of
    science and engineering.
    A substantial amount of work has been devoted to  understanding all
    possible synchronization configurations on a given network.
    In this setting, a key problem is to determine the total number of such configurations.
    Through an algebraic formulation, for tree and cycle graphs, we provide an upper bound on this number using
    the birationally invariant intersection index of a system of
    rational functions on a toric variety.
\end{abstract}

\section{Introduction}\label{sec:introduction}

The root counting problem for systems of nonlinear equations is a fundamental
problem in mathematics that has a wide range of applications.
Given an algebraic variety $X$ and complex vector spaces $L_1,\dots,L_n$
of rational functions on $X$, it has been established by
K. Kaveh and A.G. Khovanskii~\cite{kaveh_newton_2009}
that for generic choices $f_1 \in L_1,\dots,f_n \in L_n$,
the number of common complex roots of $f_1,\dots,f_n$ in $X$ is a fixed number,
known as the \emph{birationally invariant intersection index} of
$L_1,\dots,L_n$ in $X$, denoted $[L_1,\dots,L_n]$.
Moreover, $[L_1,\dots,L_n]$ is given by the mixed volume of Newton-Okounkov bodies
associated with $L_1,\dots,L_n$ and hence a far generalization of the well known
BKK bound~\cite{bernshtein_number_1975,khovanskii_newton_1978,kushnirenko_newton_1976}.
Computation of $[L_1,\dots,L_n]$ remains difficult.
This paper focuses on the indirect computation of this index for
an algebraic formulation of the ``Kuramoto equations''  rooted in
the study of spontaneous synchronization in networks of connected oscillators
which is a ubiquitous phenomenon that has been discovered and studied
in a wide range of disciplines including physics, biology, chemistry,
and engineering~\cite{dorfler_synchronization_2014}.
Mathematically, a network of $N=n+1$ oscillators can be described
by a weighted graph $G = (V,E,A)$ in which vertices $V = \{0,\dots,n\}$
represent the oscillators, edges $E$ represent their connections, and
weights $A = \{a_{ij}\}$ represent the \emph{coupling strength} along edges.
In isolation, the oscillators have their own natural frequency
$\omega_0,\dots,\omega_n$.
However, in a network of oscillators the tug of war between the oscillators' tendency
to oscillate in their own natural frequencies and the influence of their neighbors
gives rise to rich and complicated phenomenon.
This is captured by the Kuramoto model~\cite{Kuramoto1975}
\begin{equation}
    \frac{d \theta_i}{dt} =
    \omega_i -
    \sum_{j \in \mathcal{N}_G(i)} a_{ij} \sin(\theta_{i}-\theta_{j})
    \quad \text{ for } i = 0,\dots,n
    \label{equ:kuramoto-ode}
\end{equation}
where each $\theta_i \in [0,2\pi)$ is the phase angle that describes the status
of the $i$-th oscillator,
and $\mathcal{N}_G(i)$ is the set of neighbors of the $i$-th vertex.
A configuration $\boldsymbol{\theta} = (\theta_0,\dots,\theta_n)$
is said to be in \emph{frequency synchronization}
if $\frac{d\theta_i}{dt} = 0$ for all $i$ at $\boldsymbol{\theta}$.
To remove the inherent degree of freedom given by uniform rotations,
it is customary to fix $\theta_0 = 0$.
Then such synchronization configurations are characterized by the
system of $n$ nonlinear equations
\begin{equation}
    \omega_i - \sum_{j \in \mathcal{N}_G(i)} a_{ij} \sin(\theta_{i}-\theta_{j}) = 0
    \quad \text{ for } i = 1,\dots,n
    \label{equ:sync-sin}
\end{equation}
in the variables $\theta_1,\dots,\theta_n$ with constant $\theta_0 = 0$.
Then, the root counting problem is:
\begin{problem}[Real solution count]\label{prb:original}
    Given $\omega_1,\dots,\omega_n \in \R$ and a weighted graph of $n+1$ nodes,
    what is the maximum number of real solutions
    the induced system~\eqref{equ:sync-sin} could have?
\end{problem}

An upper bound to this answer, that is independent from network topology,
is shown to be $\binom{2n}{n}$~\cite{Baillieul1982}.
However, recent studies~\cite{Chen2016,Mehta2015}
suggests much tighter upper bounds that are sensitive to network topology
may exist.
In this paper, we show that this is true.

To leverage tools from algebraic geometry, we shall reformulate
the synchronization system~\eqref{equ:sync-sin} 
as a system of rational equations.
Using the identity
$\sin(\theta_{i} - \theta_{j}) = \frac{1}{2\imag}
(e^{ \imag(\theta_{i} - \theta_{j})} - e^{-\imag(\theta_{i} - \theta_{j})})$
where $\imag = \sqrt{-1}$,~\eqref{equ:sync-sin} can be transformed into
\begin{equation*}
    \omega_{i} -
    \sum_{j \in \mathcal{N}_G(i)}
    \frac{a_{i,j}}{2\imag} (
    e^{ \imag \theta_{i}} e^{-\imag \theta_{j}} -
    e^{-\imag \theta_{i}} e^{ \imag \theta_{j}}
    ) = 0
    \quad \text{ for } i = 1,\dots,n.
\end{equation*}
With the substitution $x_{i} := e^{\imag \theta_{i}}$ for $i = 1,\dots,n$,
we obtain the Laurent polynomial system
\begin{equation}
    \label{equ:sync-laurent}
    F_{G,i}(x_1,\dots,x_n) = \omega_{i} - \sum_{j \in \mathcal{N}_G(i)} a_{ij}'
    \left(
        \frac{x_i}{x_j} - \frac{x_j}{x_i}
    \right) = 0
    \quad \text{ for } i = 1,\dots,n
\end{equation}
where $a_{ij}' = \frac{a_{ij}}{2\imag}$ and $x_0 = 1$ is a constant.
This system, $F_G = (F_{G,1},\dots,F_{G,n})$, is a system of $n$ rational equations
in the $n$ complex variables $\mathbf{x} = (x_1,\dots,x_n)$.
Since $x_i$'s appear in the denominator positions,
$F_G$ is only defined on ${(\C^*)}^n = {(\C \setminus \{0\})}^n$.
Clearly, each equivalence class of real solutions of~\eqref{equ:sync-sin}
(modulo translations by multiples of $2\pi$)
corresponds to a solution of~\eqref{equ:sync-laurent} in ${(\C^*)}^n$.
Therefore, we can consider a more general root counting problem:

\begin{problem}[$\C^*$-solution count problem]\label{prb:Cstar}
    Given nonzero constants $\omega_1,\dots,\omega_n$ and a weighted graph
    of $n+1$ nodes with weights $\{ a_{ij}' \}$,
    what is the maximum number of isolated $\C^*$-solutions
    the system~\eqref{equ:sync-laurent} could have?
\end{problem}

Clearly, every answer for Problem~\ref{prb:Cstar} would provide
an upper bound for the answers for Problem~\ref{prb:original}.
However, the algebraic formulation for Problem~\ref{prb:Cstar} allows the
use of powerful tools from complex algebraic geometry, in particular,
the theory of birationally invariant intersection index
which states that the maximum number of isolated solutions coincide with
the ``generic'' number of isolated solutions of an appropriate
family of systems:
For each vertex $i=1,\dots,n$, define the complex vector space of rational functions
\begin{equation}
    L_{G,i} = \operatorname{span} \left(
        \{1\} \; \cup \;
        {\{ x_i x_j^{-1} - x_i^{-1} x_j \}}_{j\in\mathcal{N}_G(i)}
        \right).
    \label{equ:L-space}
\end{equation}
With this construction, the $i$-th equation in~\eqref{equ:sync-laurent}
is an element in $L_{G,i}$.
Therefore, the number of $\C^*$-solutions of~\eqref{equ:sync-laurent}
for generic choices of weights and constant terms will be equal to
the number of common roots of $n$ generic elements from $L_{G,1},\dots,L_{G,n}$
respectively within the toric variety ${(\C^*)}^n$.
This is precisely the birationally invariant intersection index
\cite{kaveh_newton_2009}, denoted $[L_{G,1},\dots,L_{G,n}]$.
\begin{problem}[Birationally invariant intersection index]\label{prb:index}
   Given a graph $G$ with $n+1$ vertices $0,1,\dots,n$,
   let $L_{G,i} = \operatorname{span} \left(
   \{1\} \; \cup \;
   {\{ x_i x_j^{-1} - x_i^{-1} x_j \}}_{j\in\mathcal{N}(i)}
   \right)$.
   What is $[\; L_{G,1}\,,\, \dots \, ,\, L_{G,n} \;]$~?
\end{problem}

Though the intersection index $[L_{G,1},\dots,L_{G,n}]$ can be expressed
as the generalized mixed volume of the Newton-Okounkov bodies associated with
$L_{G,1},\dots,L_{G,n}$,
its direct computation, in general, remains a difficult problem.
Using a construction known as the ``adjacency polytope bound''
developed in~\cite{chen_unmixing_2017,chen_network_2015},
the primary contribution of this paper is the computation of explicit formulas
for the birationally invariant intersection index $[L_{G,1},\dots,L_{G,n}]$
for certain graphs.
In particular, we show that for trees and cycles of $N$ vertices,
the intersection index $[L_{G,1},\dots,L_{G,n}]$ is
$2^{N-1}$ and $N \binom{N-1}{\lfloor (N-1)/2 \rfloor}$ respectively.
Both are significantly less than the only known upper bound $\binom{2N-2}{N-1}$
for the general case (hetergeneous oscillators with nonuniform coupling)
of the Kuramoto equations~\eqref{equ:sync-sin} even for small values of $n$.
Asymptotically, in both cases, the ratio between the new bounds and
$\binom{2N-2}{N-1}$ goes to zero as $N \to \infty$.
Moreover, the intersection index derived from tree graphs also coincide with
the well known lower bound of the number of \emph{real} solutions to the
original (non-algebraic) system~\eqref{equ:sync-sin}
showing that the intersection index derived from a complex solution bound
can actually be attained by real solutions.
These results are dramatic improvements over the existing bound on the
number of synchronization configuration for a Kuramoto model.
They also confirm the crucial role network topology plays in the
exhaustive study of synchronization in Kuramoto model.
From a computational view point, these generically exact explicit upper bounds
on the number of solutions are also of great importance in numerical methods
for finding all synchronization configurations for the Kuramoto model:
It provides an explicit stopping criteria for iterative solvers such as
Newton-based solvers as well as the homotopy-based Monodromy method~\cite{Duff2016}.
The secondary contribution is the general approach of computing
the birationally invariant intersection index by finding the appropriate relaxation:
Using the much simpler construction of the adjacency polytope bound,
the problem is transformed into a problem of computing normalized volumes
for certain polytopes.

The rest of the paper is structured as follows.
In \S\ref{sec:background}, we briefly review the Kuramoto model
and existing results on the number of possible equilibria.
\S\ref{sec:prelim} reviews notations and well known theorems to be used.
In \S\ref{sec:tree} and \S\ref{sec:cycle},
we compute $[L_{G,1},\dots,L_{G,n}]$ for trees and cycles respectively.

\section{Kuramoto model and synchronization equations}\label{sec:background}

\begin{wrapfigure}[12]{r}{0.30\textwidth}
    \centering
    \begin{tikzpicture}[scale=0.4]
        \draw[densely dotted,color=gray] (0,0) circle(5.0);
        \draw[decoration={segment length=2.0mm,amplitude=1mm,coil},decorate] (5,0) -- (0,5);
        \draw[decoration={segment length=2.5mm,amplitude=1mm,coil},decorate] (0,5) -- (-3,-4);
        \draw[decoration={segment length=1.5mm,amplitude=1mm,coil},decorate] (5,0) -- (3,-4);
        \draw[decoration={segment length=1.5mm,amplitude=1mm,coil},decorate] (3,-4) -- (-3,-4);
        \draw[->,thick,color=black] ( 5, 0) arc[radius=5, start angle=0  , end angle=20];
        \draw[->,thick,color=black] ( 0, 5) arc[radius=5, start angle=90 , end angle=110];
        \draw[->,thick,color=black] (-3,-4) arc[radius=5, start angle=233, end angle=213];
        \draw[->,thick,color=black] ( 3,-4) arc[radius=5, start angle=307, end angle=327];
        \filldraw[thick,fill=cyan]  ( 5, 0) circle(0.4); 
        \filldraw[thick,fill=cyan]  ( 0, 5) circle(0.4); 
        \filldraw[thick,fill=cyan]  ( 3,-4) circle(0.4); 
        \filldraw[thick,fill=cyan]  (-3,-4) circle(0.4); 
    \end{tikzpicture}
    \caption{A spring network}\label{fig:spring-network}
\end{wrapfigure}
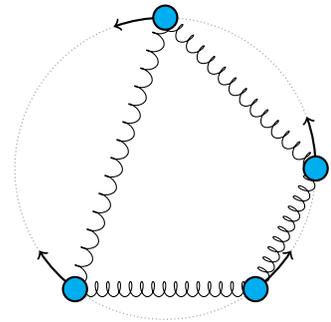

The study of synchronization in networks of coupled oscillators
is a particularly pervasive subject in a wide range of independent fields of study
in biology, physics, chemistry, engineering, and social science.
We refer to~\cite{blekhman_synchronization_1988,Acebron2005}
for a detailed historical account for this topic.
The simplest mechanical analog of the coupled oscillator model~\eqref{equ:sync-sin}
is a spring network, shown in Figure~\ref{fig:spring-network}, that consists of
a set of weightless particles constrained to move on the unit circle without
friction or collision~\cite{dorfler_synchronization_2014}.
Here, the coupling strength\footnote{In the original model proposed by Kuramoto,
the coupling strengths are symmetric, i.e., $a_{ij} = a_{ji}$.
However, in more general applications (such as power-flow equations),
perfect symmetry may not hold.}
$a_{ij} > 0$ characterizes the stiffness
of the spring connecting particles $i$ and $j$,
and $\frac{d\theta_i}{dt}$ represents the angular velocity
(or equivalently, frequency) of the $i$-th particle.
Of great interest is the configuration in which the angular velocity of
\emph{all} particles can become perfectly aligned, known as
\emph{frequency synchronization}.
That is, $\frac{d\theta_i}{dt} = c$ for $i=0,\dots,n$ and a constant $c$.
Adopting a rotational frame of reference, we can always assume $c = 0$.
That is, frequency synchronization configurations are equivalent to equilibria
of the Kuramoto model~\eqref{equ:kuramoto-ode}.
Under this assumption, the $n+1$ equilibrium equations must sum to zero.
This allows the elimination of one of the equations,
producing the system~\eqref{equ:sync-sin} of $n$ equations in $n$ unknowns.
Despite its mechanical origin, the frequency synchronization system~\eqref{equ:sync-sin}
naturally appears in a long list of seemingly unrelated fields,
including electrical power networks~\cite{Baillieul1982,dorfler_2013},
flocking behavior in biology and control theory~\cite{Justh2004,Vicsek1995},
and decentralized clock synchronization~\cite{Simeone2008}.
We refer to~\cite{dorfler_synchronization_2014} for a detailed list.

In~\cite{Baillieul1982}, an upper bound on the number of equilibria of the
Kuramoto model (solutions to~\eqref{equ:sync-sin}) induced by a graph of
$N$ vertices with any coupling strengths is shown to be $\binom{2N-2}{N-1}$.
For certain cases such as the Kuramoto model on the one, two and three-dimensional
lattice graphs with different boundary conditions,
as well as for complete and planar graphs,
all or at least a class of equilibria were
analytically~\cite{Casetti:June2003:0022-4715:1091,delabays2016multistability,delabays2017multistability,kastner2011stationary,mehta2011stationary,Nerattini:2012pi,ochab2010synchronization,xin2016analytical}
and numerically~\cite{Hughes:2012hg,hughes2014inversion,manik2016cycle,Mehta:2013iea,xi2017synchronization}
found in previous studies.
For tree graphs of $N$ nodes, it is well known that there could be as many as
$2^{N-1}$ real equilibria.
Various algebraic formulations have been used to leverage results from
algebraic geometry and numerically find some or all equilibria for certain
small graphs~\cite{Chen2016,mehta2011finding,Mehta2015,Mehta:2009zv}.
Recently, in the special case of ``rank-one coupling'', i.e., the matrix $[a_{ij}]$
has rank 1, a much smaller bound $2^N-2$ was established~\cite{Coss2017}.
Based on the theory of the BKK bound,
a search for topology-dependent bounds on the number of solutions
to~\eqref{equ:sync-sin} and~\eqref{equ:sync-laurent}
was initiated in~\cite{Chen2016,Mehta2015}.
In the present contribution, we provide explicit formulas for a much stronger
solution bound: the birationally invariant intersection index.

\section{Preliminaries and notations}\label{sec:prelim}

For a compact set $Q \subset \R^n$, $\operatorname{vol}_n(Q)$ denotes its
standard Euclidean volume, and the quantity $n! \operatorname{vol}_n(Q)$
is its \emph{normalized volume}, denoted $\NVol_n(Q)$.
Say $Q$ is \emph{convex} if it contains the line segment
connecting any two points $Q$.
For a set $X \subset \R^n$, its \emph{convex hull} is the smallest convex set
containing it, denoted $\conv(X)$,
and its \emph{affine span} is the smallest affine subspace of $\R^n$
containing it, denoted $\aff(X)$.
A (\emph{convex}) \emph{polytope} is the convex hull of a finite set of points.
Of particular importance in the current context are convex polytopes whose
vertices lie in $\Z^n$.
Such polytopes are called \emph{lattice polytopes}.
A full dimensional convex lattice polytope $P \subset \R^n$ is said to be
\emph{reflexive} if its dual
\[
P^* = \{ \mathbf{x} \in \R^n \mid \langle \mathbf{x}, \mathbf{p} \rangle \ge -1,\; \forall \mathbf{p} \in P \}
\]
is also a lattice polytope.
Given two convex polytopes $P \subset \R^n$ and $Q \subset \R^m$
both containing the origin, their \emph{free sum}, denoted $P \oplus Q$,
is $\conv ( P' \cup Q' ) \subset \R^{n+m}$ where
\[
P' = \{ (\mathbf{p}, \mathbf{0}) \in \R^{n+m} \mid \mathbf{p} \in P \}
\]
and
\[
Q' = \{ (\mathbf{0}, \mathbf{q}) \in \R^{n+m} \mid \mathbf{q} \in Q \}.
\]
An important fact is that under mild conditions, the normalized volume of
a free sum of lattice polytopes factors.
\begin{lemma}[{\cite[Theorem 1]{BraunFreeSum}}]\label{lem:free sum-vol}
    Given two convex lattice polytopes $P$ and $Q$ both containing the origin
    as an interior point, if one of them is reflexive, then
    $\NVol(P \oplus Q) = \NVol(P) \cdot \NVol(Q)$.
\end{lemma}

The set ${(\C^*)}^n$, known as an \emph{algebraic torus}, has the structure of an
abelian group under component-wise multiplication,
and it will be the space in which we study the root count of synchronization equations.
A \emph{Laurent monomial} in $\mathbf{x}=(x_1,\dots,x_n)$ induced by vector
$\mathbf{a} = (a_1,\dots,a_n) \in \Z^n$ is the formal expression
$\mathbf{x}^{\mathbf{a}} = x_1^{a_1} \cdots x_n^{a_n}$.
It is easy to verify that as a map from ${(\C^*)}^n$ to $\C^*$,
$\mathbf{x}^{\mathbf{a}}$ is actually a character, i.e., a group homomorphism.
In general, a system of Laurent monomials induced by
$\mathbf{a}_1,\dots,\mathbf{a}_m \in \Z^n$ give rise to the group homomorphism
$\mathbf{x} \mapsto (\mathbf{x}^{\mathbf{a}_1},\dots,\mathbf{x}^{\mathbf{a}_m})$
between ${(\C^*)}^n$ and ${(\C^*)}^m$.
Of particular importance, is the case where $m=n$.
\begin{lemma}[\cite{fulton_introduction_1993}]~\label{lem:toric-automorphism}
    Given vectors $\mathbf{a}_1,\dots,\mathbf{a}_n \in \Z^n$, the map
    $\mathbf{x} \mapsto (\mathbf{x}^{\mathbf{a}_1},\dots,\mathbf{x}^{\mathbf{a}_m})$
    is an automorphism of ${(\C^*)}^n$ if and only if
    $|\det [\mathbf{a}_1,\dots,\mathbf{a}_n]| = 1$
    and in that case, the map is a bi-holomorphism.
\end{lemma}
For the integer matrix $A = [\mathbf{a}_1,\dots,\mathbf{a}_n]$
and $\mathbf{x} = (x_1,\dots,x_n)$ above, we use the compact notation
${\mathbf{x}}^A = (\mathbf{x}^{\mathbf{a}_1},\dots,\mathbf{x}^{\mathbf{a}_1})$
to represent the automorphism induced by $A$.
Such a square integer matrix $A$ with $|\det(A)| = 1$ is said to be \emph{unimodular}.
More generally, an integer matrix (not necessarily square) is
\emph{totally unimodular} if all its nonsingular submatrices are unimodular.
This concept also extend to lattice polytopes:
A lattice simplex is \emph{unimodular} if its normalized volume is 1,
and a simplicial subdivision of a lattice polytope is \emph{unimodular}
if it consists of only unimodular simplices.

A \emph{Laurent polynomial} is a finite linear combination of distinct Laurent monomials,
i.e., an expression of the form
$f = \sum_{\mathbf{a} \in S} c_{\mathbf{a}} \mathbf{x}^{\mathbf{a}}$ for some finite $S \subset \Z^n$.
The set $\conv(S) \subset \R^n$ is called the \emph{Newton polytope} of $f$.
Given a nonzero $\mathbf{v} \in \R^n$, $\init_{\mathbf{v}} f$ is defined to be
$\sum_{\mathbf{a} \in {(S)}_{\mathbf{v}}} c_{\mathbf{a}} \mathbf{x}^{\mathbf{a}}$
where ${(S)}_{\mathbf{v}}$ is the subset of $S$ on which the linear functional
$\langle \cdot, \mathbf{v} \rangle$ attains its minimum.
Extending this notation to a system of Laurent polynomials $F = (f_1,\dots,f_n)$,
we write $\init_{\mathbf{v}} F = (\init_{\mathbf{v}} f_1,\dots,\init_{\mathbf{v}} f_n$).
Newton polytopes play critical roles in calculating the generic number of isolated
solutions in ${(\C^*)}^n$ (or simply $\C^*$-solutions)
a system of $n$ Laurent polynomial equations could have.
Indeed, this generic $\C^*$-solution count is given by the mixed volume
of the Newton polytopes. This is the content of Bernshtein's
Theorem~\cite{bernshtein_number_1975,khovanskii_newton_1978,kushnirenko_newton_1976},
and this count has since been known as the \textbf{BKK bound}~\cite{canny_optimal_1991}.
Though we will not directly compute BKK bounds, the condition for BKK bound to
be exact will be used in establishing our main results.
\begin{theorem}[{\cite[Theorem B]{bernshtein_number_1975}}]\label{thm:bernshtein-b}
    Consider a system of $n$ Laurent polynomials $F = (f_1,\dots,f_n)$
    in $n$ variables.
    If $\init_{\mathbf{v}} F$ has no solution in ${(\C^*)}^n$ for any nonzero vector $\mathbf{v} \in \R^n$,
    then all solutions of $F(\mathbf{x}) = \mathbf{0}$ are isolated
    and the total number is exactly the BKK bound of the system.
\end{theorem}
An important fact is that for generic choice of the coefficients,
the BKK bound is exact.
\begin{lemma}[{\cite{bernshtein_number_1975}}]\label{lem:init-sys}
    Let $F = (f_1,\dots,f_n)$ be a system of $n$ Laurent polynomials
    in $n$ variables.
    For generic choices of coefficients, and any nonzero $\mathbf{v} \in \R^n$
    $\init_{\mathbf{v}} F$ has no solution in ${(\C^*)}^n$.
\end{lemma}
\begin{remark}\label{rmk:generic}
    ``Generic choice'' is a subtle concept in algebraic geometry.
    In the current context, it is sufficient to take the following
    ``probability one'' interpretation:
    If the coefficients are chosen at random (with independent distribution)
    among all possible complex coefficients, then with probability one,
    Lemma~\ref{lem:init-sys} is true.
    However, using the set of coefficients to parametrize a nonlinear system
    is not completely precise:
    For any nonsingular square matrix $A$, the system $F$ and $A \cdot F$
    are naturally equivalent.
    Consequently, the more precise parametrization using a certain Grassmannian
    has to be considered in order to make sense of the concept of generic choice.
\end{remark}

A relaxation of the BKK bound was developed in the context of algebraic
synchronization equations~\cite{chen_unmixing_2017}
as well as the closely related ``power-flow equations.''~\cite{chen_network_2015}.

\begin{definition}[Adjacency polytope]\label{def:adj-polytope}
    Given a graph $G$, we define its \textbf{adjacency polytope} to be
    \begin{equation*}
        \nabla_G =
        \conv(\nabla_{G,1} \cup \cdots \cup \nabla_{G,n}) =
        \conv( \{ \pm (\mathbf{e}_i-\mathbf{e}_j) \mid (i,j) \in E(G) \}).
    \end{equation*}
    The normalized volume $\NVol(\nabla_G)$ is called the
    \textbf{adjacency polytope bound} of $G$.
\end{definition}

The polytope $\nabla_G$ can be considered as a geometric encoding
of the topology of the graph $G$.
Adjacency polytopes have been previously studied in order to identify properties of a related semigroup algebra, such as in \cite{centrallysymmetric}.
However, previous work has not addressed the normalized volume of these polytopes.
A simple observation~\cite{chen_unmixing_2017,chen_network_2015} is that
the adjacency polytope bound (or simply, \textbf{AP bound}) is indeed an
upper bound for answers of Problem~\ref{prb:Cstar} and~\ref{prb:index}:


\begin{proposition}\label{pro:ap-bound}
    Given a graph $G$ containing vertices $\{0,1,\dots,n\}$,
    the number of isolated $\C^*$-solutions for the
    algebraic system~\eqref{equ:sync-laurent} is bounded by
    the AP bound $\NVol(\nabla_G)$.
\end{proposition}

By comparing the constructions of the solution bounds outlined above,
it is easy to verify the following chain of inequalities
\begin{equation}
    \parbox{12ex}{\centering $\R$-solution\\ count of~\eqref{equ:sync-sin}}
    \;\le \;
    \parbox{12ex}{\centering $\C^*$-solution\\ count of~\eqref{equ:sync-laurent}}
    \; \le \;
    [L_{G,1},\dots,L_{G,n}]
    \; \le \;
    \parbox{7ex}{\centering BKK\\ bound} \; \le \;
    \parbox{7ex}{\centering AP\\ bound} 
    \label{equ:comparison}
\end{equation}

\section{Tree graphs}~\label{sec:tree}
This section provides the answers for Problem~\ref{prb:Cstar} and~\ref{prb:index}
for a tree graph $T_N$ containing $N=n+1$ vertices.
The strategy is to bound $[L_{T_N,1},\dots,L_{T_N,n}]$ from above using the AP bound,
and then bound it from below by examining the actual number of solutions.
With this, we shall show $[L_{T_N,1},\dots,L_{T_N,n}]$ is $2^n = 2^{N-1}$.
This agrees with a well known fact in the study of the Kuramoto model:
for tree graphs, the original (non-algebraic) Kuramoto model~\eqref{equ:kuramoto-ode}
could have as many as $2^{N-1}$ real equilibria.
This shows that even though it is derived from a complex algebraic formulation,
the bound $[L_{T_N,1},\dots,L_{T_N,n}]$ on the number of complex solutions
can be attained by just real solutions.
That is, the algebraization~\eqref{equ:sync-laurent} of~\eqref{equ:sync-sin}
and the extension to the field of complex numbers does not significantly alter
the geometry of the underlying problem.

For a vertex $i$ in $T_N$, let $\pi(i)$ be the unique parent vertex of $i$, let
$\sigma(i)$ be the set of all descendant nodes of $i$,
and let $d(i)$ be the depth of the vertex $i$.

\begin{lemma}\label{lem:tree-transform}
    The map $\phi = (\phi_1,\dots,\phi_n) : {(\C^*)}^n \to {(\C^*)}^n$ given by
    \begin{equation*}
        \phi_i(y_1,\dots,y_n) = y_i \, \prod_{k=1}^{d(i)-1} y_{\pi^k(i)}
        \quad\text{for } i=1,\dots,n
    \end{equation*}
    is a bijection, and the Jacobian matrix $D\phi$ is nonsingular everywhere.
\end{lemma}

\begin{proof}
    A tree, by definition, has no cycles, so it is always possible
    to re-index the vertices such that vertex 0 is the root
    and $\pi(i) < i$ for any $i$.
    With this convention, we can write $\phi$ as $\phi(\mathbf{y}) = y^A$
    where $y=(y_1,\dots,y_n)$, and $A$ is an $n \times n$ upper triangular
    integer matrix with all diagonal entries being 1.
    Then $A$ is a unimodular matrix and hence $A^{-1}$ is also a unimodular
    integer matrix.
    It is easy to verify that $\psi(\mathbf{x}) = \mathbf{x}^{A^{-1}}$
    is an inverse of $\phi$, and therefore they are both bijections.
    Moreover, since $\det A = 1$, by Lemma~\ref{lem:toric-automorphism},
    $D\phi(\mathbf{y})$ is nonsingular for all $\mathbf{y} \in {(\C^*)}^n$.
\end{proof}

Being a bijection, the transformation $\phi$ given in Lemma~\ref{lem:tree-transform}
preserves the solution count of any system of equations.
Moreover, since $D\phi$ remains nonsingular on ${(\C^*)}^n$, $\phi$ also preserves
the more subtle local structures at each solution including multiplicities
and local dimensions.

\begin{theorem}
    For a tree graph $T_N$ consisting of $N$ nodes, the Adjacency Polytope bound
    of the induced algebraic system~\eqref{equ:sync-laurent}
    is $2^{N-1}$.
\end{theorem}

This result agrees with the general analysis from recent studies
\cite{dekker_synchronization_2013,dorfler_synchronization_2014}.
A similar result for the root counting problem for power-flow equations
has been developed in~\cite{Guo1990}.

\begin{proof}
    Let $F_{T_N}(\mathbf{x}) = F_{T_N}(x_1,\dots,x_n)$ be the
    algebraic system~\eqref{equ:sync-laurent} induced by
    the tree graph $T_N$.
    Then each non-constant monomial in $F_{T_N}(\mathbf{x})$ must be of the form
    $x_i x_{\pi(i)}^{-1}$ or $x_i^{-1} x_{\pi(i)}$ for some $i \in \{1,\dots,n\}$.
    With the substitution given by $x_i = \phi_i(y_1,\dots,y_n)$ for $i=1,\dots,n$,
    as defined in the previous lemma, it is easy to verify that
    \[
        x_i x_{\pi(i)}^{-1} =
        \left(
            y_i \, \prod_{k=1}^{d(i)-1} y_{\pi^k(i)}
        \right)
        \left(
            y_{\pi(i)}^{-1} \, \prod_{k=1}^{d(\pi(i))-1} y_{\pi^k(\pi(i))}^{-1}
        \right)
        = y_i.
    \]
    Therefore the set of monomials which appear in $F_{T_N}(\phi(\mathbf{y}))$
    is exactly the set $\{1\} \cup \{ y_1,\dots,y_n \} \cup \{ y_1^{-1},\dots,y_n^{-1} \}$.
    Under the same transformation, the Adjacency Polytope becomes the cross-polytope
    \[
        \conv \left( \bigcup_{i=1}^n \conv( \{\pm \mathbf{e}_i\} ) \right),
    \]
    which is a free sum of the $n$ line segments
    $\conv( \{\pm \mathbf{e}_i\} )$ for $i=1,\dots,n$.
    By Lemma~\ref{lem:free sum-vol}, the normalized volume of this polytope
    is the product of the normalized volume of each of the summands.
    Since each line segment is of length 2,
    the AP bound is therefore $2^n = 2^{N-1}$.
\end{proof}

We now show the AP bound is actually attainable.
That is, there exist choices of complex values for
$\{a_{ij}'\}$ and $\omega_1,\dots,\omega_n$ in~\eqref{equ:sync-laurent}
for which the system has exactly $2^n$ isolated $\C^*$-solutions.

\begin{lemma}\label{lem:tree-reduction}
    For the tree graph $T_N$ containing $N = n+1$ vertices,
    the induced algebraic system~\eqref{equ:sync-laurent}
    is equivalent to the system
    \begin{equation}
        \omega^*_i
        -
        a_{i,\pi(i)}' \,
        \left( \frac{x_i}{x_{\pi(i)}} - \frac{x_{\pi(i)}}{x_i} \right)
        = 0
        \quad \text{ for } i = 1,\dots,n
        \label{equ:tree-reduced}
    \end{equation}
    for some complex constants $\omega^*_1,\dots,\omega^*_n$.
\end{lemma}
Here, the equivalence means the two systems have the same solution set in $\C^*$.

\begin{proof}
    For $N=2$, the system~\eqref{equ:sync-laurent} contains only one equation
    \[
        \omega_i - a_{1,0}' (x_1 x_0^{-1} - x_1^{-1} x_0) = 0
    \]
    where $x_0 = 1$. The statement is obviously true in this case.

    Now consider a tree $T_N$ consisting of $N = n+1$ nodes and
    assume the statement is true for any tree of smaller sizes.
    Fixing any leaf vertex in the tree, without loss, we can re-index the vertices
    so that this leaf vertex has index $n$ and its unique parent vertex is $n-1$.
    In this arrangement, the $n$-th (last) equation in~\eqref{equ:sync-laurent} is
   \begin{equation}
       \omega_n - a_{n,n-1}' (x_n x_{n-1}^{-1} - x_{n-1} x_n^{-1}) = 0,
       \label{equ:leaf-x}
   \end{equation}
    while the $(n-1)$-th equation is
   \begin{equation}
       \omega_{n-1} -
       a_{n-1,n}' (x_{n-1} x_n^{-1} - x_n x_{n-1}^{-1}) -
       \sum_{j \in \mathcal{N}(n-1) \setminus \{n\}}
       a_{n-1,j}' (x_{n-1} x_j^{-1} - x_j x_{n-1}^{-1})
       = 0.
       \label{equ:leaf-parent}
   \end{equation}
    Then adding $a_{n-1,n}'/a_{n,n-1}'$ times~\eqref{equ:leaf-x}
    to~\eqref{equ:leaf-parent} produces
   \begin{equation*}
       \left(\omega_{n-1} + \frac{a_{n-1,n}'}{a_{n,n-1}'} \omega_n \right)
       -
       \sum_{j \in \mathcal{N}(n-1) \setminus \{n\}}
       a_{n-1,j}' \left(\frac{x_{n-1}}{x_j} - \frac{x_j}{x_{n-1}} \right)
       = 0.
   \end{equation*}
   With this transformation, the first $n-1$ equations do not involve $x_n$
   and form a smaller algebraic system induced by a tree graph
   consisting of $n$ vertices $0,1,\dots,n-1$.
   By the induction hypothesis, this smaller system can be transformed into
   the desired form given in~\eqref{equ:tree-reduced} without altering
   the solution set.
   By induction, the statement is true for all tree graphs.
\end{proof}

\begin{lemma}
    Given a tree graph $T_N$ containing $N$ vertices, there exist choices of
    complex valued weights ${\{a_{ij}'\}}_{(i,j) \in E(T_N)}$
    and complex constants $\omega_1,\dots,\omega_n$,
    such that the induced system $F_{T_N}(\mathbf{x}) = \mathbf{0}$
    has exactly $2^{N-1}$ nonsingular isolated $\C^*$-solutions.
\end{lemma}

\begin{proof}
    By Lemma~\ref{lem:tree-reduction}, the induced algebraic system
    $F_T(\mathbf{x})$ is equivalent to~\eqref{equ:tree-reduced}.
    Under the transformation $\mathbf{x} = \phi(\mathbf{y})$ given in
    Lemma~\ref{lem:tree-transform}, $F_{T_N}(\phi(\mathbf{y}))$ is
    \begin{equation}
        \omega^*_i -
        a_{i,\pi(i)}' \,
        ( y_i - y_i^{-1} )
        = 0
        \quad \text{ for } i = 1,\dots,n
        \label{equ:tree-reduced-y}
    \end{equation}
    which has the same number of isolated nonsingular solutions in ${(\C^*)}^n$
    as the original system.
    Concerning $\C^*$-solutions, the $i$-th equation in the above system is
    equivalent to the quadratic equation
    \[
        \omega^*_i
        -
        a_{i,\pi(i)}' \, y_i^2
        +
        a_{i,\pi(i)}'
        = 0
    \]
    which has exactly two $\C^*$-solutions for generic choice of coefficients
    (even if we require $a_{ij}' = a_{ji}'$).
    Since there are $n$ independent quadratic equations in $y_1,\dots,y_n$ respectively,
    the generic root count for~\eqref{equ:tree-reduced-y} is exactly $2^n = 2^{N-1}$.
    Consequently the $\C^*$-solution count of the original system $F_{T_N}$
    can also reach $2^{N-1}$.
\end{proof}

\begin{theorem}\label{cor:index-tree}
    Given a tree graph $T_N$ containing $N = n+1$ vertices,
    let $L_{T_N,1},\dots,L_{T_N,n}$ be the subspace of rational functions
    defined in~\eqref{equ:L-space}.
    Then
    \[
        [\, L_{T_N,1}, \dots, L_{T_N,n}\, ] = 2^n = 2^{N-1}
    \]
\end{theorem}

\begin{proof}
    By~\eqref{equ:comparison}, $[L_{T_N,1},\dots,L_{T_N,n}]$ is trapped
    in between the $\C^*$-solution count of~\eqref{equ:sync-laurent}
    and its AP bound.
    We have shown both to be $2^{N-1}$.
    Therefore we can conclude $[ L_{T_N,1}, \dots, L_{T_N,n}] = 2^{N-1}$.
\end{proof}

    By carefully keeping track of the transformation of coefficients,
    it is possible to argue that the ``generic solution count''
    for~\eqref{equ:tree-reduced-y} and the original system~\eqref{equ:sync-laurent}
    are actually the same,
    thereby establishing Corollary~\ref{cor:index-tree} directly.
    However, as noted in Remark~\ref{rmk:generic}, the concept of
    ``generic coefficients'' is more subtle than it may appear.
    Therefore, here we prefer the straightforward calculation of the AP bound
    over such genericity argument.
\section{Cycle Graphs}\label{sec:cycle}

In the study of the Kuramoto model, cycle graphs may be considered as basic
building blocks as recent works suggests that it is plausible that
detailed analysis of the local geometry near equilibria can be done on a
cycle-by-cycle basis~\cite{Bronski2015}.
In the context of power-flow study, the analysis of the Kuramoto model on
cycle graphs is also of great practical importance~\cite{xi2017synchronization}.

For a cycle graph $C_N$ of $N = n + 1$ vertices (labeled by $\{0,\dots,n\}$),
we shall show the intersection index $[L_{C_N,1},\dots,L_{C_N,n}]$ is
$(n+1)\binom{n}{\lfloor n/2 \rfloor}$.
Following the same strategy used in the previous section,
we first compute the AP bound for the cycle graph $C_N$.
Then we show there is no gap between $[L_{C_N,1},\dots,L_{C_N,n}]$ and the AP bound.

The set of edges is $E(C_N) = \{(0,1),(1,2),\dots,(n-1, n),(n,0)\}$.
The induced adjacency polytope (Definition~\ref{def:adj-polytope}) is
\[
    \nabla_{C_N} = \conv \{\pm(\mathbf{e}_i - \mathbf{e}_j) \mid (i,j) \in E(C_N)\},
\]
where $\mathbf{e}_0 = (0,\dots,0)$ as before.
The AP bound for $F_{C_N}$ is defined to be the normalized volume of $\nabla_{C_N}$;
thus, the first goal of this section will be to identify this normalized volume.
It will be simplest to first notice that $\nabla_{C_N}$ is unimodularly equivalent
to the polytope
\[
    P_N = \conv \{
        \pm \mathbf{e}_1 \, , \,\dots \, , \, \pm \mathbf{e}_n\,,\;
        \pm (\mathbf{e}_1 + \cdots + \mathbf{e}_n)
        \}.
\]
Such an equivalence can be seen by applying the normalized volume-preserving transformation
given by
\[
    \begin{bmatrix}
        1 & 0 & 0 & \cdots & 0 \\
        1 & 1 & 0 & \cdots & 0 \\
        1 & 1 & 1 & \cdots & 0 \\
        \vdots & \vdots & \vdots & \ddots & \vdots \\
        1 & 1 & 1 & \cdots & 1 \\
    \end{bmatrix}
\]
to each vertex of $\nabla_{C_N}$.
One reason that this is desirable is that it becomes clear $P_N$ is totally unimodular,
that is, the matrix formed by placing the vertices of $P_N$ as the columns is a
totally unimodular matrix.
Since $\mathbf{0}$ is the average of all vertices of $P_N$,
it is an interior point of $P_N$.
Thus, a unimodular triangulation of the boundary of $P_N$ will induce a
unimodular triangulation of $P_N$ itself, where the simplices are of the form $\conv\{\mathbf{0} \cup \Delta \}$ where $\Delta$ is a simplex in the triangulation of the boundary.
This will be our strategy, since the number of simplices in a unimodular triangulation of a polytope
is identical to the normalized volume of the polytope.


When $N$ is odd, then $P_N$ is called a \emph{del Pezzo} polytope.
In this case, it is known~\cite{Nill} that $P_N$ is simplicial,
that is, every facet is a simplex.
Together with $P_N$ being totally unimodular, its normalized volume is therefore
equal to the number of its facets, which was shown to be $N\binom{N-1}{(N-1)/2}$.
So, we only need to consider when $N$ is even and would like to obtain an analogous formula.

\begin{proposition}
    For even $N$, let $\Lambda_n \subseteq {\{-1,1\}}^n$ be the set of sequences
    $\boldsymbol{\lambda} = (\lambda_1, \dots, \lambda_n)$ such that
    $\sum_{i=1}^n \lambda_i = 1$.
    The facets of $P_N$ are then
    \[
        \mathcal{F}(P_N) =
        \left\{
            \pm \conv \left\{
                \lambda_1 \mathbf{e}_1, \dots, \lambda_n \mathbf{e}_n,
                \sum_{i=1}^n \mathbf{e}_i
            \right\}
            \mid (\lambda_1,\dots,\lambda_n) \in \Lambda_n
        \right\}
    \]
\end{proposition}

\begin{proof}
    First, observe that the vertices of a facet must consist of a subset of
    \begin{equation}\label{eq: sign pattern}
    	\{
            \lambda_1 \mathbf{e}_1,
            \dots,
            \lambda_n \mathbf{e}_n,
            \lambda_{n+1} ( \mathbf{e}_1 + \cdots + \mathbf{e}_n )
        \}
    \end{equation}
    for some choice of $\lambda_1,\dots,\lambda_{n+1} \in \{-1,1\}$.
    Otherwise, two vertices $\pm \mathbf{v}$ of $P_N$ would be part of a facet,
    which is impossible since the line segment $\conv \{ -\mathbf{v},\mathbf{v}\}$
    passes through the interior of $P_N$.

    Next, note that if $F$ is a facet, then so is $-F$ since $P_N = -P_N$.
    One specific choice of facet is
    \[
        F_0 =
        \conv
        \left \{
            \mathbf{e}_1, \dots, \mathbf{e}_{(n-1)/2},-\mathbf{e}_{(n+1)/2},
            \dots,
            -\mathbf{e}_n, -\sum_{i=1}^n \mathbf{e}_i
        \right \}.
    \]
    To see why this is true, observe that each of the vertices in $F_0$ lies on the hyperplane
    \[
    	\{(x_1\dots,x_n) \in \R^n \mid \ell(x_1,\dots,x_n) = 1\}
    \]
    where
    \[
    	\ell(x_1,\dots,x_n) = \sum_{i = 1}^{(n-1)/2} x_i - \sum_{i = (n+1)/2}^n x_i,
    \]
    and all other vertices $\mathbf{v}$ of $P_N$ satisfy $\ell(\mathbf{v}) = -1$.
    Moreover, the first $n$ vertices defining $F_0$ are clearly affinely independent,
    so $\dim F_0 = n-1$.
    Therefore, $F_0$ is indeed a facet of $P_N$.

    Any other choice of $(\lambda_1,\dots,\lambda_n) \in \Lambda_n$ for the
    elements in $\mathcal{F}(P_N)$ will result in a facet as well,
    since the resulting convex hull is unimodularly equivalent to $F_0$.
    Hence, the same arguments can be applied to these sets.
    It remains to show that no other set of vertices will form a facet.

    Take any element of~\eqref{eq: sign pattern} such that there are $k \geq 2$
    more negative coefficients on the summands $\mathbf{e}_1,\dots,\mathbf{e}_n$
    than positive coefficients, and set $\lambda_{n+1} = -1$.
    Without loss of generality, we can assume
    $\lambda_1 = \dots = \lambda_{(n-2k+1)/2} = 1$ and the remaining $\lambda_i = -1$.
    Call their convex hull $F'$.
    Form $\ell(x)$ as before and note that the first $n$ vertices of $F'$ satisfy
    $\ell(x_1,\dots,x_n) = 1$.
    Additionally, the vertices $\mathbf{v}$ of $P_N$ not in $F'$ satisfy
    $\ell(x_1,\dots,x_n) < 1$.
    However, $\ell(-\mathbf{e}_1 - \cdots - \mathbf{e}_n) = k$,
    so $\aff(F')$ actually passes through the interior of $P_N$ and
    cannot define a facet.

    Note as well that if we take any $n$-element subset of~\eqref{eq: sign pattern}
    without $\pm(\mathbf{e}_1+\cdots+\mathbf{e}_n)$,
    then we come across a similar problem as in the previous paragraph.
    If we take an $n$-element subset that excludes $\pm \mathbf{e}_j$ for some $j$,
    then the resulting hyperplane is exactly the same as if we included $\pm \mathbf{e}_j$.
    Therefore, there are no facets of any other form.
\end{proof}

By permuting the coordinates of each facet of $P_N$,
the elements of $\mathcal{F}(P_N)$ are all unimodularly equivalent to each other.
Additionally, it is clear that
\[
    |\mathcal{F}(P_N)| = 2\binom{N-1}{N/2 - 1}
\]
since each $\lambda \in \Lambda_n$ corresponds to a unique facet of $P_N$
containing $\mathbf{e}_1 + \cdots + \mathbf{e}_n$.
There are $\binom{n}{(n-1)/2}$ elements in $\Lambda_n$,
and this must be doubled to account for the facets containing
$-(\mathbf{e}_1 + \cdots + \mathbf{e}_n)$.

In order to determine the number of simplices in a unimodular triangulation
of $P_N$, we now only need to compute the number of simplices in a unimodular triangulation of a facet.
For convenience, we will select the facet $F_0$ from the previous proof.
Applying the unimodular matrix transformation $x \mapsto Ax$, where $A = (a_{i,j})$ is the $n \times n$ matrix
\[
	a_{i,j} = \begin{cases}
			1 & \text { if } i = j \text{ or both } i = n,\, j < (n-1)/2 \\
			-1 & \text{ if both } i = n,\, (n-1)/2 < j < n\\
			0 & \text{ else }
		\end{cases}
\]
we obtain a polytope whose vertices are identical to those of $F_0$ in the first
$n-1$ coordinates and are exactly $1$ in the final coordinate.
This allows us to consider $\overline{F}_0$, the projection of $f(F_0)$ to the
first $n-1$ coordinates.
As a result, we have
\[
	\overline{F}_0 =
    \conv
    \left \{
        \mathbf{0},
        \mathbf{e}_1,
        \dots,
        \mathbf{e}_{(n-1)/2},
        -\mathbf{e}_{(n+1)/2},
        \dots,
        -\mathbf{e}_{n-1},
        -\sum_{i=1}^{n-1} \mathbf{e}_i
    \right \}
    \subseteq \R^{n-1}.
\]

Notice that we can write $\overline{F}_0 = \conv \{G_1 \cup G_2\}$, where
\[
	G_1 = \conv \left \{
        \mathbf{e}_1,\dots, \mathbf{e}_{(n-1)/2}, -\sum_{i=1}^{n-1} \mathbf{e}_i
    \right \}
\]
and
\[
	G_2 = \conv\{(0,- \mathbf{e}_{(n+1)/2},\dots,- \mathbf{e}_{n-1}\}.
\]
Moreover, the intersection of their affine spans is a single point
\[
	\{ \mathbf{v}_0 \} =
    \aff(G_1) \cap \aff(G_2) =
    \left \{
    \left(
        0,\dots,0,-\frac{1}{(n+1)/2},\dots,-\frac{1}{(n+1)/2}
    \right)
    \right \}.
\]
The lattices generated by $\aff(G_i) \cap \Z^{n-1}$ and $\mathbf{v}_0$, after translating by $-\mathbf{v}_0$, are
\[
	L_1 = \Z(\mathbf{e}_1 - v_0,\dots,\mathbf{e}_{(n-1)/2} - \mathbf{v}_0, -(\sum_{i=1}^{n-1} \mathbf{e}_i) - \mathbf{v}_0),
\]
and
\[
	L_2 = \Z(-\mathbf{v}_0, -\mathbf{e}_{(n+1)/2} - \mathbf{v}_0,\dots,-\mathbf{e}_{n-1} - \mathbf{v}_0) = \Z(\mathbf{v}_0,-\mathbf{e}_{(n+1)/2},\dots,-\mathbf{e}_{n-1}),
\]
where $\Z A$ indicates the set of $\Z$-linear combinations of elements of $A$.
The lattices $L_1$ and $L_2$ are \emph{complementary},
meaning they intersect only at $\mathbf{0}$,
and each point of $L = \Z(\mathbf{e}_1,\dots,\mathbf{e}_{n-1},\mathbf{v}_0)$ is a sum of a unique element from $L_1$ and a unique element from $L_2$.

Together, these facts mean $\overline{F}_0$ is the \emph{affine free sum} of
$G_1$ and $G_2$, as introduced in~\cite{BeckJayawantMcAllister}.
Since $G_2$ is a standard simplex, its normalized volume is $1$; moreover, it is known that the $k$-dimensional standard simplex $\Delta_k$ is Gorenstein of index $k+1$, that is, there exists a unique vector $v \in \Z^k$ (namely, $v = (-1,\dots,-1)$) such that the polar dual of $\Delta_k' = (k+1)\Delta_k + v$, defined as
\[
    \{x \in \R^k \mid x^Ty \leq 1 \text{ for all } y \in \Delta_k' \},
\]
is also a lattice polytope.
By~\cite[Corollary 5.9]{BeckJayawantMcAllister} and~\cite[Corollary 3.21]{BeckRobinsCCD2ed}, we have
\[
	\NVol(\overline{F}_0) = \NVol(G_1)\NVol(G_2) = \NVol(G_2).
\]
Therefore it remains to find the normalized volume of $G_2$, which is unimodularly equivalent to the simplex
\[
	\conv\{\mathbf{e}_1,\dots, \mathbf{e}_{(n-1)/2}, -(\mathbf{e}_1 + \cdots + \mathbf{e}_{(n-1)/2)}\}.
\]
It is straightforward to compute that this simplex has a normalized volume of $\frac{n-1}{2} + 1 = \frac{n+1}{2}$.
Connecting this argument back to our original goal, we have proven the following.

\begin{proposition}
	The normalized volume of each facet of $P(C_N)$ is $\frac{N}{2}$.
\end{proposition}

This gives us the final piece we need.

\begin{theorem}
    For a cycle graph of $N$ vertices, the adjacency polytope bound
    of~\eqref{equ:sync-laurent} is
    \[
        N \binom{N-1}{\lfloor (N-1)/2 \rfloor}.
    \]
\end{theorem}

\begin{proof}
	We already saw that the conclusion holds for when $N$ is odd.
	When $N$ is even, we now simply count
	\[
		\NVol(\overline{F}_0)|\mathcal{F}(P_N)| = \left(\frac{N}{2}\right)2\binom{N-1}{N/2 - 1} = N\binom{N-1}{\lfloor (N-1)/2 \rfloor},
	\]
	as desired.
\end{proof}


By the inequalities~\eqref{equ:comparison}, the AP bound above is also
an upper bound for the birationally invariant intersection index.
We now show there is no gap between the two.
Let $F_{C_N} = (F_{C_N,1},\dots,F_{C_N,n})$ with each $F_{C_N,i}$
being a generic element from $L_{C_N,i}$,
and let $\nabla_{C_N,1},\dots,\nabla_{C_N,n}$ be their Newton polytopes respectively.
The BKK bound
of the system $F_{C_N}$ coincides with its AP bound by \cite[Proposition 1]{chen_unmixing_2017}.
We shall significantly strengthen this statement by showing that even though
the spaces $L_{C_N,1}, \dots, L_{C_N,n}$ are not generated by monomials,
the intersection index $[L_{C_N,1}, \dots, L_{C_N,n}]$ still agrees with
the BKK bound for $F_{C_N}$.
This is done by examining the initial systems of $F_{C_N}$.
In particular, we show that even though there are algebraic relations
among the coefficients for terms in $F_{C_N}$,
such relations will not appear in any nontrivial initial systems.

\begin{theorem}
    Given a cycle graph $C_N$ containing $N = n+1$ vertices,
    let $L_{C_N,1},\dots,L_{C_N,n}$ be the subspace of rational functions
    defined in~\eqref{equ:L-space}.
    Then
    \[
        [\, L_{C_N,1}, \dots, L_{C_N,n}\, ] = N \binom{N-1}{\lfloor (N-1)/2 \rfloor}
    \]
\end{theorem}

\begin{proof}
    Let $\mathbf{v}$ be a vector in $\R^n$ such that
    ${(\nabla_{C_N,i})}_{\mathbf{v}}$ is not singleton for any $i=1,\dots,n$.
    Since the polytopes $\nabla_{C_N,i}$ all contain the origin, we must have
    \[
        \Ht_{\mathbf{v}}(\nabla_{C_N,i}) :=
        \min \{
            \langle \mathbf{x}, \mathbf{v} \rangle
            \mid \mathbf{x} \in \nabla_{C_N,i}
        \} \le 0 \quad \text{for all } i.
    \]
    If $\operatorname{ht}_{\mathbf{v}}(\nabla_{C_N,i}) = 0$ for all $i$
    then $\langle \pm (\mathbf{e}_i - \mathbf{e}_j), \mathbf{v} \rangle = 0$
    for any pair of $(i,j) \in E(C_N)$.
    It is then easy to verify that $\mathbf{v} = \mathbf{0}$.

    Now, supposing $\mathbf{v} \ne \mathbf{0}$, there must be a vertex
    $i \in \{1,\dots,n\}$ for which
    $\operatorname{ht}_{\mathbf{v}}(\nabla_{C_N,i}) < 0$.
    Recall that $\nabla_{C_N,i}$ has at most four vertices:
    $\{ \pm (\mathbf{e}_i - \mathbf{e}_j),  \pm (\mathbf{e}_k - \mathbf{e}_j) \}$
    where $\{j,k\} = \mathcal{N}_{C_N}(i)$.
    But $\langle \bullet, \mathbf{v} \rangle$ must attain negative values for
    at least two points in this set.
    That means there are exactly two points
    $\mathbf{b}_j \in \{ \mathbf{e}_i-\mathbf{e}_j, \mathbf{e}_j-\mathbf{e}_i \}$
    and
    $\mathbf{b}_k \in \{ \mathbf{e}_i-\mathbf{e}_k, \mathbf{e}_k-\mathbf{e}_i \}$
    such that
    $\langle \mathbf{b}_j, \mathbf{v} \rangle < 0$ and
    $\langle \mathbf{b}_k, \mathbf{v} \rangle < 0$.
    However, since $\mathbf{b}_j \in \nabla_{C_N,j}$
    and $\mathbf{b}_k \in \nabla_{C_N,k}$,
    $\operatorname{ht}_{\mathbf{v}}(\nabla_{C_N,j})$ and
    $\operatorname{ht}_{\mathbf{v}}(\nabla_{C_N,k})$ are both negative.
    In other words, if $\operatorname{ht}_{\mathbf{v}}(\nabla_{C_N,i}) < 0$
    for some vertex $i$,
    then $\operatorname{ht}_{\mathbf{v}}(\nabla_{C_N,j}) < 0$
    for any $j \in \mathcal{N}_{C_N}(i)$.
    Since $C_N$ is connected, as this implication propagates through the graph,
    we can conclude that $\operatorname{ht}_{\mathbf{v}}(\nabla_{C_N,j}) < 0$
    for all $i \in \{1,\dots,n\}$.
    Consequently, for each $(i,j) \in E(C_N)$, the two points
    $\mathbf{e}_i - \mathbf{e_j}$ or $\mathbf{e}_j - \mathbf{e}_i$ cannot both
    be in ${(\nabla_{C_N,i})}_{\mathbf{v}}$ or ${(\nabla_{C_N,j})}_{\mathbf{v}}$.
    Recall that $\mathbf{e}_i - \mathbf{e}_j$ and $\mathbf{e}_j - \mathbf{e}_i$
    are the exponent vectors of $\frac{x_i}{x_j}$ and $\frac{x_j}{x_i}$ respectively.
    Therefore either $\frac{x_i}{x_j}$ or $\frac{x_j}{x_i}$ appear in
    $\init_{\mathbf{v}} F_{C_N}$, but not both.
    Consequently, monomials appearing in $\init_{\mathbf{v}} F_{C_N}$
    all have independent coefficients.
    Then by Lemma~\ref{lem:init-sys}, for generic choice of coefficients,
    the initial system $\init_{\mathbf{v}} (F_{C_N}) = \mathbf{0}$
    has no solution in $(\C^*)^n$.
    This is true for any nonzero vector $\mathbf{v}$,
    so by Theorem~\ref{thm:bernshtein-b}, the number of solutions
    $F_{C_N} = \mathbf{0}$ has in ${(\C^*)}^n$ is exactly the BKK bound.
    Since $F_{C_N}$ is a generic choice,
    we can conclude that $[L_{C_N,1}, \dots, L_{C_N,n}]$ agrees with
    the BKK bound and hence the AP bound shown above.
\end{proof}





\bibliographystyle{siamplain}
\bibliography{nobodies,kuramoto,conv,bkk,chen}

\begin{thebibliography}{10}

\bibitem{Acebron2005}
{\sc J.~A. Acebr{\'{o}}n, L.~L. Bonilla, C.~J. {P{\'{e}}rez Vicente},
  F.~Ritort, and R.~Spigler}, {\em {The Kuramoto model: A simple paradigm for
  synchronization phenomena}}, Reviews of Modern Physics, 77 (2005),
  pp.~137--185, \url{https://doi.org/10.1103/RevModPhys.77.137}.

\bibitem{Baillieul1982}
{\sc J.~Baillieul and C.~Byrnes}, {\em {Geometric critical point analysis of
  lossless power system models}}, IEEE Transactions on Circuits and Systems, 29
  (1982), pp.~724--737, \url{https://doi.org/10.1109/TCS.1982.1085093}.

\bibitem{BeckJayawantMcAllister}
{\sc M.~Beck, P.~Jayawant, and T.~B. McAllister}, {\em Lattice-point generating
  functions for free sums of convex sets}, J. Combin. Theory Ser. A, 120
  (2013), pp.~1246--1262, \url{https://doi.org/10.1016/j.jcta.2013.03.007},
  \url{http://dx.doi.org/10.1016/j.jcta.2013.03.007}.

\bibitem{BeckRobinsCCD2ed}
{\sc M.~Beck and S.~Robins}, {\em Computing the continuous discretely},
  Undergraduate Texts in Mathematics, Springer, New York, second~ed., 2015,
  \url{https://doi.org/10.1007/978-1-4939-2969-6},
  \url{http://dx.doi.org/10.1007/978-1-4939-2969-6}.
\newblock Integer-point enumeration in polyhedra, With illustrations by David
  Austin.

\bibitem{bernshtein_number_1975}
{\sc D.~N. Bernshtein}, {\em {The number of roots of a system of equations}},
  Functional Analysis and its Applications, 9 (1975), pp.~183--185.

\bibitem{blekhman_synchronization_1988}
{\sc I.~I. Blekhman}, {\em {Synchronization in science and technology}},
  American Society of Mechanical Engineers, 1988.

\bibitem{BraunFreeSum}
{\sc B.~Braun}, {\em An {E}hrhart series formula for reflexive polytopes},
  Electron. J. Combin., 13 (2006), pp.~Note 15, 5,
  \url{http://www.combinatorics.org/Volume_13/Abstracts/v13i1n15.html}.

\bibitem{Bronski2015}
{\sc J.~C. Bronski, L.~DeVille, and T.~Ferguson}, {\em {Graph Homology and
  Stability of Coupled Oscillator Networks}},  (2015), pp.~1--18,
  \url{http://arxiv.org/abs/1508.01507},
  \url{https://arxiv.org/abs/1508.01507}.

\bibitem{canny_optimal_1991}
{\sc J.~Canny and J.~M. Rojas}, {\em {An optimal condition for determining the
  exact number of roots of a polynomial system}}, in Proceedings of the 1991
  {\{}International{\}} {\{}Symposium{\}} on {\{}Symbolic{\}} and
  {\{}Algebraic{\}} {\{}Computation{\}}, {\{}ISSAC{\}} '91, New York, NY, USA,
  1991, ACM, pp.~96--102, \url{https://doi.org/10.1145/120694.120707},
  \url{http://doi.acm.org/10.1145/120694.120707}.

\bibitem{Casetti:June2003:0022-4715:1091}
{\sc L.~Casetti, M.~Pettini, and E.~G.~D. Cohen}, {\em Phase transitions and
  topology changes in configuration space}, Journal of Statistical Physics, 111
  (June 2003), pp.~1091--1123(33),
  \url{http://www.ingentaconnect.com/content/klu/joss/2003/00000111/F0020005/00462874}.

\bibitem{chen_unmixing_2017}
{\sc T.~Chen}, {\em {Unmixing the mixed volume computation}}, arXiv:1703.01684
  [math],  (2017), \url{http://arxiv.org/abs/1703.01684}.

\bibitem{Chen2016}
{\sc T.~Chen, J.~Marecek, D.~Mehta, and M.~Niemerg}, {\em {A Network Topology
  Dependent Upper Bound on the Number of Equilibria of the Kuramoto Model}},
  arXiv:1603.05905 [nlin],  (2016), \url{http://arxiv.org/abs/1603.05905}.

\bibitem{chen_network_2015}
{\sc T.~Chen and D.~Mehta}, {\em {On the Network Topology Dependent Solution
  Count of the Algebraic Load Flow Equations}}, arXiv:1512.04987 [cs, math],
  (2015).

\bibitem{Coss2017}
{\sc O.~Coss, J.~D. Hauenstein, H.~Hong, and D.~K. Molzahn}, {\em {Locating and
  counting equilibria of the Kuramoto model with rank one coupling}},  (2017),
  \url{http://www3.nd.edu/{~}jhauenst/preprints/chhmKuramoto.pdf}.

\bibitem{dekker_synchronization_2013}
{\sc A.~Dekker and R.~Taylor}, {\em {Synchronization Properties of Trees in the
  Kuramoto Model}}, SIAM Journal on Applied Dynamical Systems, 12 (2013),
  pp.~596--617, \url{https://doi.org/10.1137/120899728}.

\bibitem{delabays2016multistability}
{\sc R.~Delabays, T.~Coletta, and P.~Jacquod}, {\em Multistability of
  phase-locking and topological winding numbers in locally coupled kuramoto
  models on single-loop networks}, Journal of Mathematical Physics, 57 (2016),
  p.~032701.

\bibitem{delabays2017multistability}
{\sc R.~Delabays, T.~Coletta, and P.~Jacquod}, {\em Multistability of
  phase-locking in equal-frequency kuramoto models on planar graphs}, Journal
  of Mathematical Physics, 58 (2017), p.~032703.

\bibitem{dorfler_synchronization_2014}
{\sc F.~D{\"{o}}rfler and F.~Bullo}, {\em {Synchronization in complex networks
  of phase oscillators: A survey}}, Automatica, 50 (2014), pp.~1539--1564,
  \url{https://doi.org/10.1016/j.automatica.2014.04.012}.

\bibitem{dorfler_2013}
{\sc F.~D{\"{o}}rfler, M.~Chertkov, and F.~Bullo}, {\em {Synchronization in
  complex oscillator networks and smart grids.}}, Proceedings of the National
  Academy of Sciences of the United States of America, 110 (2013),
  pp.~2005--10, \url{https://doi.org/10.1073/pnas.1212134110},
  \url{http://www.ncbi.nlm.nih.gov/pubmed/23319658}.

\bibitem{Duff2016}
{\sc T.~Duff, C.~Hill, A.~Jensen, K.~Lee, A.~Leykin, and J.~Sommars}, {\em
  Solving polynomial systems via homotopy continuation and monodromy},  (2016),
  \url{http://arxiv.org/abs/1609.08722},
  \url{https://arxiv.org/abs/1609.08722}.

\bibitem{fulton_introduction_1993}
{\sc W.~Fulton}, {\em {Introduction to toric varieties}}, no.~131, Princeton
  University Press, 1993.

\bibitem{Guo1990}
{\sc S.~Guo and F.~Salam}, {\em {Determining the solutions of the load flow of
  power systems: Theoretical results and computer implementation}}, dec 1990,
  pp.~1561--1566 vol.3, \url{https://doi.org/10.1109/CDC.1990.203876}.

\bibitem{Hughes:2012hg}
{\sc C.~Hughes, D.~Mehta, and J.-I. Skullerud}, {\em {Enumerating Gribov copies
  on the lattice}}, Annals Phys., 331 (2013), pp.~188--215,
  \url{https://doi.org/10.1016/j.aop.2012.12.011},
  \url{https://arxiv.org/abs/1203.4847}.

\bibitem{hughes2014inversion}
{\sc C.~Hughes, D.~Mehta, and D.~J. Wales}, {\em An inversion-relaxation
  approach for sampling stationary points of spin model hamiltonians}, The
  Journal of chemical physics, 140 (2014), p.~194104.

\bibitem{Justh2004}
{\sc E.~Justh and P.~Krishnaprasad}, {\em {Equilibria and steering laws for
  planar formations}}, Systems {\&} Control Letters, 52 (2004), pp.~25--38,
  \url{https://doi.org/10.1016/j.sysconle.2003.10.004},
  \url{http://linkinghub.elsevier.com/retrieve/pii/S0167691103002949}.

\bibitem{kastner2011stationary}
{\sc M.~Kastner}, {\em Stationary-point approach to the phase transition of the
  classical xy chain with power-law interactions}, Physical Review E, 83
  (2011), p.~031114.

\bibitem{kaveh_newton_2009}
{\sc K.~Kaveh and A.~G. Khovanskii}, {\em {Newton-Okounkov bodies, semigroups
  of integral points, graded algebras and intersection theory}},  (2009),
  \url{http://arxiv.org/abs/0904.3350}, \url{https://arxiv.org/abs/0904.3350}.

\bibitem{khovanskii_newton_1978}
{\sc A.~G. Khovanskii}, {\em {Newton polyhedra and the genus of complete
  intersections}}, Functional Analysis and Its Applications, 12 (1978),
  pp.~38--46, \url{https://doi.org/10.1007/BF01077562},
  \url{http://dx.doi.org/10.1007/BF01077562}.

\bibitem{Kuramoto1975}
{\sc Y.~Kuramoto}, {\em {Self-entrainment of a population of coupled non-linear
  oscillators}}, Lecture Notes in Physics, Springer Berlin Heidelberg, 1975,
  pp.~420--422, \url{http://link.springer.com/chapter/10.1007/BFb0013365}.

\bibitem{kushnirenko_newton_1976}
{\sc A.~G. Kushnirenko}, {\em {Newton polytopes and the {\{}Bezout{\}}
  theorem}}, Functional Analysis and Its Applications, 10 (1976), pp.~233--235,
  \url{https://doi.org/10.1007/BF01075534},
  \url{http://link.springer.com/article/10.1007/BF01075534}.

\bibitem{manik2016cycle}
{\sc D.~Manik, M.~Timme, and D.~Witthaut}, {\em Cycle flows and multistabilty
  in oscillatory networks: an overview}, arXiv preprint arXiv:1611.09825,
  (2016).

\bibitem{mehta2011finding}
{\sc D.~Mehta}, {\em Finding all the stationary points of a potential-energy
  landscape via numerical polynomial-homotopy-continuation method}, Physical
  Review E, 84 (2011), p.~025702.

\bibitem{Mehta2015}
{\sc D.~Mehta, N.~S. Daleo, F.~D{\"{o}}rfler, and J.~D. Hauenstein}, {\em
  {Algebraic geometrization of the Kuramoto model: Equilibria and stability
  analysis}}, Chaos: An Interdisciplinary Journal of Nonlinear Science, 25
  (2015), p.~053103, \url{https://doi.org/10.1063/1.4919696},
  \url{http://scitation.aip.org/content/aip/journal/chaos/25/5/10.1063/1.4919696}.

\bibitem{Mehta:2013iea}
{\sc D.~Mehta, C.~Hughes, M.~Schr\"ock, and D.~J. Wales}, {\em {Potential
  Energy Landscapes for the 2D XY Model: Minima, Transition States and
  Pathways}}, J. Chem. Phys., 139 (2013), p.~194503,
  \url{https://doi.org/10.1063/1.4830400},
  \url{https://arxiv.org/abs/1311.5859}.

\bibitem{mehta2011stationary}
{\sc D.~Mehta and M.~Kastner}, {\em Stationary point analysis of the
  one-dimensional lattice landau gauge fixing functional, aka random phase xy
  hamiltonian}, Annals of Physics, In Press (2011), pp.~--.

\bibitem{Mehta:2009zv}
{\sc D.~Mehta, A.~Sternbeck, L.~von Smekal, and A.~G. Williams}, {\em {Lattice
  Landau Gauge and Algebraic Geometry}}, PoS, QCD-TNT09 (2009), p.~025,
  \url{https://arxiv.org/abs/0912.0450}.

\bibitem{Nerattini:2012pi}
{\sc R.~Nerattini, M.~Kastner, D.~Mehta, and L.~Casetti}, {\em {Exploring the
  energy landscape of XY models}}, Phys.Rev., E87 (2013), p.~032140,
  \url{https://doi.org/10.1103/PhysRevE.87.032140},
  \url{https://arxiv.org/abs/1211.4800}.

\bibitem{Nill}
{\sc B.~Nill}, {\em Classification of pseudo-symmetric simplicial reflexive
  polytopes}, in Algebraic and geometric combinatorics, vol.~423 of Contemp.
  Math., Amer. Math. Soc., Providence, RI, 2006, pp.~269--282,
  \url{https://doi.org/10.1090/conm/423/08082},
  \url{http://dx.doi.org/10.1090/conm/423/08082}.

\bibitem{ochab2010synchronization}
{\sc J.~Ochab and P.~Gora}, {\em Synchronization of coupled oscillators in a
  local one-dimensional kuramoto model}, Acta Physica Polonica. Series B,
  Proceedings Supplement, 3 (2010), pp.~453--462.

\bibitem{centrallysymmetric}
{\sc H.~Ohsugi and T.~Hibi}, {\em Centrally symmetric configurations of integer
  matrices}, Nagoya Math. J., 216 (2014), pp.~153--170,
  \url{https://doi.org/10.1215/00277630-2857555},
  \url{http://dx.doi.org/10.1215/00277630-2857555}.

\bibitem{Simeone2008}
{\sc O.~Simeone, U.~Spagnolini, Y.~Bar-Ness, and S.~Strogatz}, {\em
  {Distributed synchronization in wireless networks}}, IEEE Signal Processing
  Magazine, 25 (2008), pp.~81--97,
  \url{https://doi.org/10.1109/MSP.2008.926661},
  \url{http://ieeexplore.ieee.org/document/4607217/}.

\bibitem{Vicsek1995}
{\sc T.~Vicsek, A.~Czir{\'{o}}k, E.~Ben-Jacob, I.~Cohen, and O.~Shochet}, {\em
  {Novel Type of Phase Transition in a System of Self-Driven Particles}},
  Physical Review Letters, 75 (1995), pp.~1226--1229,
  \url{https://doi.org/10.1103/PhysRevLett.75.1226},
  \url{https://link.aps.org/doi/10.1103/PhysRevLett.75.1226}.

\bibitem{xi2017synchronization}
{\sc K.~Xi, J.~L. Dubbeldam, and H.~X. Lin}, {\em Synchronization of cyclic
  power grids: Equilibria and stability of the synchronous state}, Chaos: An
  Interdisciplinary Journal of Nonlinear Science, 27 (2017), p.~013109.

\bibitem{xin2016analytical}
{\sc X.~Xin, T.~Kikkawa, and Y.~Liu}, {\em Analytical solutions of equilibrium
  points of the standard kuramoto model: 3 and 4 oscillators}, in American
  Control Conference (ACC), 2016, IEEE, 2016, pp.~2447--2452.

\end{thebibliography}

\end{document}